\documentclass[11pt,reqno]{amsart}

\usepackage[utf8]{inputenc}
\usepackage{amsfonts,latexsym,amssymb,amsmath,amsthm,amsrefs,etex,slashed,cancel}
\usepackage{todonotes}
\usepackage{hyperref}
\usepackage{a4wide}
\usepackage{enumerate}
\usepackage{mathrsfs}
\usepackage{MnSymbol}
\usepackage{nicefrac}

\hypersetup{
    colorlinks,
    linkcolor={red!50!black},
    citecolor={blue!50!black},
    urlcolor={blue!80!black}
}

\usepackage{graphicx}
\usepackage[a4paper]{geometry}
\usepackage{amssymb}
\usepackage{amsmath}
\usepackage{amsthm}
\usepackage{verbatim}
\usepackage{color}
\usepackage{url}
\usepackage{fancyhdr}

 \newtheorem{theorem}{Theorem}
 \newtheorem{lemma}{Lemma}
 
 \newtheorem{prop}{Proposition}

 \theoremstyle{definition}
 \newtheorem{definition}{Definition}
 
 \newtheorem{remark}{Remark}

\allowdisplaybreaks
\renewcommand{\Re}{\operatorname{Re}}
\renewcommand{\Im}{\operatorname{Im}}

\DeclareMathOperator{\supp}{supp}

\newcommand{\rmi}{\mathrm{i}}

\newcommand{\calO}{\mathcal{O}}

\newcommand{\R}{\mathbb{R}}
\newcommand{\C}{\mathbb{C}}

\newcommand{\calEO}{\mathcal{E}(\Omega)}

\numberwithin{equation}{section}

\begin{document}

\title{Analytic properties of heat equation solutions and reachable sets}

\author{Alexander Strohmaier}
\address{School of Mathematics\\ University of Leeds\\ Leeds\\ Yorkshire\\ LS2 9JT, UK}
\email{a.strohmaier@leeds.ac.uk}
\author{Alden Waters} 
\address{Bernoulli Institute\\ Rijksuniversiteit Groningen\\ Nijenborgh 9\\ 9747 AG Groningen\\ Netherlands}
\email{a.m.s.waters@rug.nl}

\thanks{}

\date{\today}
 
\begin{abstract}
There recently has been some interest in the space of functions on an interval satisfying the heat equation for positive time in the interior of this interval. Such functions were characterised as being analytic on a square with the original interval as its diagonal. In this short note we provide a direct argument that the analogue of this result holds in any dimension. For the heat equation on a bounded Lipschitz domain $\Omega \subset \R^d$  at positive time all solutions are analytically extendable to a geometrically determined subdomain $\mathcal{E}(\Omega)$  of $\mathbb{C}^d$ containing $\Omega$. 
This domain is sharp in the sense that there is no larger domain for which this is true. If $\Omega$ is a ball we prove an almost converse of this theorem. Any function that is analytic in an open neighbourhood of $\mathcal{E}(\Omega)$ is reachable in the sense that it can be obtained from a solution of the heat equation at positive time. This is based on an analysis of the convergence of heat equation solutions in the complex domain using the boundary layer potential method for the heat equation. The converse theorem is obtained using a Wick rotation into the complex domain that is justified by our results. This gives a simple explanation for the shapes appearing in the one-dimensional analysis of the problem in the literature. It also provides a new short and conceptual proof in that case. 
\end{abstract}

\maketitle

\section{Introduction and background}
The questions of control and reachability for the heat equation have a long history. Null controllability of the heat equation has been extensively researched since the $70's$. The one dimensional case was closely examined in the pioneering work of \cite{HO} using biorthogonal families. Sharp characterisations of null-controllability were obtained in the $d$-dimensional case using elliptic Carleman estimates \cite{lr} or parabolic Carleman estimates \cite{FO}, c.f. also \cites{EZ, rosier, zuazua, DZZ, CZ}. 

By contrast, much less is known about exact controllability of the heat equation. Theorems in \cites{MRR, SD, HT} contain characterisations of the reachable set for the one dimensional case in terms of analytic functions. For arbitrary domains (non-empty open connected sets) $\Omega$ in $\mathbb{R}^d$ the question of characterisation of the reachable set, here denoted as $\mathcal{R}_{\Omega}$, remains an open question. We seek to answer the question ``what are the properties of $\mathcal{R}_{\Omega}$ when the domain $\Omega$ is a bounded subdomain of $\mathbb{R}^d$?"

\subsection{The reachable set $\mathcal{R}_\Omega$}

For an open bounded Lipschitz domain $\Omega \subset \R^d$ we consider the heat equation 
\begin{align}
&\partial_tu =\Delta u \quad \mathrm{in} \quad [0,T]\times \Omega \nonumber \\& \label{INg}
u|_{t=0}=u_0\quad \mathrm{in} \quad \Omega \\& \nonumber
u|_{[0,T] \times \partial\Omega}=h \nonumber
\end{align}
with initial data $u_0$ and boundary data $h$. This problem is well posed in various function spaces. For the sake of concreteness we take
initial data $u_0 \in H^1_0(\Omega)$
on the time interval $[0,T]$, and  with $u \in H^{1,\frac{1}{2}}([0,T]\times \Omega)$ and  $h \in H^{\frac{1}{2},\frac{1}{4}}([0,T]\times \partial \Omega)$ in certain mixed Sobolev spaces that are described in
section \ref{thermal}. As discussed in Section \ref{thermal} this problem is well-posed (Prop. \ref{wpwp}) and since the heat operator is hypoelliptic the solution is necessarily in $C^\infty((0,T] \times  \Omega)$.
Generically, the set
\begin{align*}
\mathcal{R}_{\Omega}(T,u_0)=\{ v \in C^\infty(\Omega) \mid \, v(x) = u(T,x)\, \textrm{where } u(t,x) \textrm{ solves} \; \eqref{INg} \textrm{ for some } h \in H^{\frac{1}{2},\frac{1}{4}} \}
\end{align*}
is referred to as the {\sl reachable set}. 
By null-controllabilty of the heat equation with boundary controls we have
$\mathcal{R}_{\Omega}(T,u_0) = \mathcal{R}_{\Omega}(T,0)$ and that $\mathcal{R}_{\Omega}(T,0)$ does not depend on $T>0$.
Indeed, a nonzero function $u_0(x)$ with zero Dirichlet boundary conditions is controllable to $0$ after any finite time $T$ and can therefore be subtracted off from the problem \label{IN} by linearity. Boundary null controllability for the heat equation for smooth domains can be found in \cites{rosier,lr}. In \cite{apraiz} null controllability for the heat equation on certain Lipschitz domains (including smooth domains) using $L^{\infty}$ boundary controls is established, and also in the older \cite{FO} for $C^2$ domains. The paper \cite{CZ} also explains how the results from \cite{FO} provide some subspaces of the reachable space, in a close spirit to \cite{HO}.
We remark here that boundary null controllability for the heat equation with
$H^{\frac{1}{2},\frac{1}{4}}([0,T] \times \partial \Omega)$ boundary controls is an easy consequence of null-controllability with $L^\infty$-control on a slightly larger open domain that contains the closure of $\Omega$. Indeed, let $\Omega'$ be an open set with smooth boundary such that $\Omega \Subset \Omega'$. Extending the initial value by zero to $\Omega'$ we can find 
boundary controls for $\Omega'$ such that the solution vanishes at $t=T$ and with boundary controls that are smooth near $t=0$. The latter can always be achieved by solving the heat equation using the heat kernel of $\R^d$ for small times, then using $L^\infty$-boundary controls for the remaining time interval.
The restriction to $\Omega$ then gives a solution on $Q$ that vanishes at $t=T$ with boundary controls that are in  $H^{\frac{1}{2},\frac{1}{4}}([0,T] \times \partial \Omega)$ (see Section \ref{thermal}). The so constructed boundary controls in fact exhibit much higher regularity, but we will not discuss this here.

Therefore the following definition makes sense.

\begin{definition} \nonumber
 Let $\Omega$ be a bounded Lipschitz domain in $\R^d$.
 The {\sl reachable set} $\mathcal{R}_{\Omega} \subset C^\infty(\Omega)$ 
 is defined as $\mathcal{R}_{\Omega}(T,0)$ for some (and hence for all) $T>0$.
\end{definition}
\noindent
Due to the linearity of the problem, the reachable set $\mathcal{R}_{\Omega}$ is a vector space.
  \\

Our goal is to describe the analytic properties of the reachable set $\mathcal{R}_{\Omega}$ 
 for bounded Lipschitz domains $\Omega\subset\mathbb{R}^d$.
 
 \subsubsection{Choice of function space}
  In the literature several other function spaces to pose the initial-boundary-value problem \eqref{INg} have been used. 
  From a modern partial differential equation perspective and for smooth domains the most natural choice is an adapted scale of mixed Sobolev spaces. For Lipschitz domains the scale needs to be restricted and for the sake of definiteness we have chosen to follow the very natural choice of \cite{costabel}, c.f. also \cite{costabel2,costabel3,costabel4, mitrea1,mitrea2, mitrea} for the theory of these spaces in relationship to boundary layer theory for low regularity domains. This space also has the advantage that the relevant estimates for thermal layer potential theory are well established. This theory provides a very explicit description of the analytic continuation of solutions of the heat equation (c.f. our proof of Prop. \ref{main1}) and we believe this observation to be interesting in its own right.
  
  Obviously other function spaces result in a different definition of the reachable set. Our main results are however robust in that they actually do not depend on the choice of function class in the problem setup as long as well-posedness and null-controllability hold and the function spaces contain the set of smooth functions (see Remark \ref{function-spaces}).\\ 
  
We now describe what is known for the one dimensional heat equation on an interval.
\subsection{The example of the heat equation on an interval}
For a finite time $T$, the one dimensional heat equation as a control problem can be stated as: 
\begin{align}\label{IN1}
&\partial_t u(t,x)-\partial_{x}^2u(t,x)=0\quad x\in (0,1) \,\, t\in [0,T]\\& \nonumber
u(t,0)=h_0(t) \quad u(t,1)=h_1(t) \quad t\in [0,T] \\& \nonumber
u(0,x)=u_0(x) \quad x\in (0,1) \nonumber
\end{align}
with $u_0\in L^2(0,1)$ and $h_0(t),h_1(t)\in L^2(0,T)$. Here equality for $L^2$ functions is understood in the usual sense as an almost everywhere equality.

A function $u$ is reachable if there exists two control inputs $h_0(t),h_1(t)$ in $L^2(0,T)$ such that the solution satisfies
\begin{align*}
u(T,x)=u_1(x) \quad \textrm{a.e. for } \,x\in (0,1)
\end{align*}
and $u(0,x)=u_0(x)$. The operator $Au=u''$ with Dirichlet boundary conditions has domain $D(A)=H^2(0,1)\cap H_0^1(0,1)\subset L^2(0,1)$. Naturally we have that 
\begin{align*}
e_n(x)=\sqrt{2}\sin(n\pi x) \quad n\geq 1
\end{align*}
with the set  $\{e_n(x)\}_{n\in\mathbb{N}}$ an orthonormal basis in $L^2(0,1)$ consisting of eigenfunctions of $A$. Decomposing a given $u\in D(A)$ as 
\begin{align*}
u(x)=\sum\limits_{n=1}^{\infty}c_ne_n(x)=\sum\limits_{n=1}^{\infty}\sqrt{2}c_n\sin(n\pi x)
\end{align*}
then it is known as a result of \cite{HO} that $u$ is necessarily reachable if we have for some $\epsilon>0$ 
\begin{align}\label{1dfourier}
\sum\limits_{n=1}^{\infty}\frac{|c_n|}{n}\exp((1+\epsilon)n\pi)<\infty. 
\end{align}
Unfortunately this last condition implies that $u$ and all of its odd derivatives vanish at $0$ and $1$. This condition is not natural and for instance excludes polynomial functions. 
For example, functions of the form 
\begin{align*}
\frac{c_d\alpha (2\pi)^{\frac{d+2}{2}}}{(4\pi^2\alpha^2+|x|^2)^{\frac{d+1}{2}}}, \quad\quad c_d=\frac{\Gamma(\frac{d+1}{2})}{\pi^{\frac{d+1}{2}}}
\end{align*}
 satisfy the conditions in Section 3. (For this function the case $d=1$ was already covered by results in \cite{MRR}). 
 
\section{Statement of the results}
 
We now need the definition of the subsets of $\mathbb{C}^d$ over which we are extending our solutions. 
\begin{definition}
For $\Omega$ an open and bounded Lipschitz domain in $\mathbb{R}^d$ we let $\mathcal{E}(\Omega)$ denote the set
\begin{align*}
&\mathcal{E}(\Omega)=\left\{ z=x+iy \in \C^d\;| \; x \in \Omega,\; |y| <\mathrm{dist}(x,\partial\Omega), y\in \mathbb{R}^d \right\}.
\end{align*}
\end{definition}

\begin{figure}[h!]
 \begin{center}
 \includegraphics*[width=7cm]{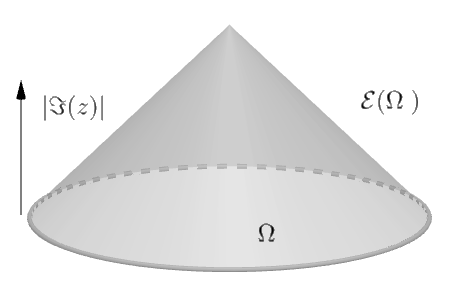}
 \end{center}
 \caption{Illustration of the domain $\mathcal{E}(\Omega)$ when $|\Im(z)|$ is projected onto the vertical axis. In one dimension this is usually depicted as a double cone, but one should note that this picture would be misleading in higher dimensions because then $\mathcal{E}(\Omega) \setminus \R^d$ is connected when $\Omega$ is.}
\end{figure}

Of course $\calEO$ is the pre-image of the positive part of the usual domain of dependence $D^+(\Omega) \subset \R^{d+1}$ of the wave equation on $d+1$-dimensional Minkowski space under the map $\C^d \to \R^{d+1}, z = x + \rmi y \mapsto (x,|y|)$. It can be thought of as a ball bundle over $\Omega$ where the radius of the ball over the point $x \in \Omega$ is given by the distance of $x$ to $\partial \Omega$.

\noindent  \emph{Notations:} 
\begin{itemize}
\item For an open subset $U \subset \C^d$ the set of holomorphic functions on $U$ is denoted by $\calO(U)$. We endow it with the topology of uniform convergence on compact subsets of $U$.
\item For a subset $E \subset \C^d$ that is the closure of an open set we denote by $\calO(E)$ the set 
$$\calO(E)=\bigcup_{U \supset E, U \textrm{ open}} \calO(U).$$
\item A function $f \in \calO(\mathcal{E}(\Omega))$ is completely determined by its restriction to $\Omega$ and therefore we think of the set of analytic functions $\calO(\mathcal{E}(\Omega))$ as a subset of $C^\infty(\Omega)$ without further mention.
\end{itemize}

Notice in dimension 1 for $\Omega=(-L,L)$ this set $\calEO$ coincides with the definition of the set in \cite{SD}, which is a square in the complex plane $\mathbb{C}$ with $(-L,L)$ as one of its diagonals.  
The criterion in the above definition are found in Theorem 3.1.12 in \cite{hormander}. 

The first result we prove here is stated as
\begin{prop}\label{main1}
If $\Omega$ is a bounded Lipschitz domain, then $\mathcal{R}_{\Omega} \subset  \calO(\mathcal{E}(\Omega))$. The domain $\mathcal{E}(\Omega)$  is optimal in the following sense.  For any $p \in \C^d \setminus \overline{\mathcal{E}(\Omega)}$ there exists a function $u \in \mathcal{R}_{\Omega}$ which does not have an analytic extension to a connected open set containing $p$ and $\Omega$.
 \end{prop}
This thus establishes that $\calEO$ is the optimal domain all reachable functions can be extended to as a holomorphic function. 
The fact that solutions to the heat equation extend holomorphically can for example be found for more regular domains in \cite{John}*{p.219} and a simple modification of this using cutoff functions and restrictions shows that this is true for Lipschitz domains. We give here a proof which is based on thermal boundary layer theory which actually gives a more precise representation of the solution, and which may be useful for more precise investigations. In particular, in one dimension this gives rise to completely explicit expressions (see Section \ref{1dsection}). 

Our main theorem is the almost converse theorem for special geometries.
We then have: 
\begin{theorem}\label{main2}
 Suppose that $\Omega = B_R(x_0)$ is a ball. Then
 \begin{align*}
   \calO\left(\overline{\calEO}\right) \subset \mathcal{R}_{\Omega} \subset  \calO(\mathcal{E}(\Omega)).
\end{align*} 
\end{theorem}

The proof of Theorem  \ref{main2} is obtained by ``Wick rotation" which transforms the forward heat equation into the backward heat equation.
\begin{remark}
By Hartogs' extension theorem the statement of Theorem \ref{main2} can be strengthened in dimensions $d \geq 2$ to the statement 
$\calO\left( U \right) \subset \mathcal{R}_{\Omega}$ for any connected open neighbourhood $U$ of $\partial \calEO$. The set $\partial \calEO$ is connected as it contains a sphere bundle over $\Omega$ as a dense set.
\end{remark}

\begin{remark} \label{function-spaces}
 By a simple enlargement argument of Proposition \ref{main1} and Theorem \ref{main2} also hold for boundary controls in other function classes. The reason is that elements in $\calO\left(\overline{\calEO}\right)$ are reachable by a solution of the heat equation on a slightly larger domain and therefore the boundary controls on the smaller domain are automatically smooth. Hence, elements in $\calO\left(\overline{\calEO}\right)$ are reachable with smooth initial values and smooth boundary controls. Similarly, if $u$ solves the heat equation with less regular boundary controls and initial data, one can still show that $u(T)$ belongs to $\calO(\mathcal{E}(\Omega))$. This follows trivially from our results by making the domain slightly smaller. This characterisation is therefore very robust with respect to choices of function spaces.
 \end{remark}

The main technical theorem may be interesting in its own right.
\begin{theorem}\label{main3}
 Assume that $u(t,\cdot)$ is a solution of the heat equation \eqref{INg} with $u_0 \in \calO(\mathcal{E}(\Omega))$. Then 
 $u(t,\cdot)$ converges to $u_0$ as $t \to 0_+$ uniformly on compact subsets of $\mathcal{E}(\Omega)$.
\end{theorem}

We note here that the known results in dimension one are special cases of these theorems. 
The articles \cite{MRR} and \cite{SD} determine when the class of functions belonging to the reachable set in dimension one are analytically extendable and vice versa using Gevrey polynomials and the Cauchy formula respectively.  The characterisation in \cite{SD}  is a special case of our result, but the proof techniques are considerably different in that we use a simple Wick rotation and avoid Fourier analysis. The idea came from the use of the complex Gaussian frame in \cite{GW, killip}, and the use of analytic extension to the upper half plane in \cite{killip}. In these cases for the initial boundary value problem a frame of Gaussians was used to model the initial data, but these are only approximations to solutions of the heat equation. However, their analytic extensions are easily identifiable. The paper \cite{killip} also uses an analytic extension of the heat kernel to the upper-half plane to prove decay properties of solutions to the Schr\"odinger equation (under a Wick rotation) in the exterior of a smooth compact domain. The authors decided to use the exact solution to the IBVP for the heat equation using thermal layer potentials but the ideas of a Gaussian-like solution and extension to the upper half plane were borrowed from these previous papers. 

 We futhermore direct the reader to the works of \cite{orsoni, KNT, orsonihartmann, CR} which completely characterise the 1d reachable states using the Bargman transform as a different approach. In particular since an exact characterisation has been given for this reachable space in \cite{orsonihartmann} as the Bergman space of the square, which corresponds to our $\mathcal{E}(\Omega)$ in one dimension, this shows that the converse inclusion in Prop. \ref{main1} is not true. The almost converse inclusion given by Theorem \ref{main2} is the best we can do with respect to our Frechet space $\mathcal{O}(\mathcal{E}(\Omega))$. They have also extended their results to the Hermite heat equation in \cite{orsonihartmann2} in 1d. We expect that some different analysis of the heat kernel in higher dimensions is needed to expand the results perhaps in terms of wave-front sets in order to get sharp results in higher dimensions.  The reader is also invited to see \cites{rosier2, lissy1, lissy2, coron1,coron2, mu, MLG} for related results in the 1d case. Recently \cite{ELBT} strengthens the results of the present work, showing that the reachable space is stable under a small perturbation. The range of the backwards heat operator for $\mathbb{R}^d$ has also been characterised in \cite{hall, hall2} which is less general than our Theorem \ref{main2} and \cite{zelditch, zelditch2, zelditch3} with applications to ergodic theory. Our focus is on some short control theory applications. 

The paper is structured as follows. To keep the article self-contained we start by giving the background on boundary layer potential theory for Lipschitz domains in Section \ref{thermal}. This is a summary of results from \cite{costabel}. 
Section \ref{s23} gives the proofs of Proposition \ref{main1} and Theorem \ref{main2} assuming the validity of Theorem  \ref{main3}. Theorem  \ref{main3} is then proved in Section \ref{s5}.

\section{Thermal boundary layer potential theory}\label{thermal}
Let $\Omega \subset \R^d$ be a bounded Lipschitz domain with $d\geq 1$ and boundary $\partial\Omega$. In other words, the boundary $\partial\Omega$ is locally congruent to the graph of a Lipschitz continuous function on $\mathbb{R}^{d-1}$, and $\Omega$ is located on exactly one side of the boundary. 
For a number $T>0$ which is fixed we write
\begin{align*}
Q= (0,T)\times\Omega \quad \Sigma=(0,T)\times\partial\Omega.
\end{align*}
Further we let $\Omega_t=\{t\}\times\Omega$ so that 
\begin{align*}
\partial Q=\overline{\Sigma}\cup \Omega_0\cup\Omega_T.
\end{align*}
For $r,s\geq 0$ we let
\begin{align*}
H^{r,s}(\mathbb{R}\times\mathbb{R}^d)=L^2(\mathbb{R};H^r(\mathbb{R}^d))\cap H^s(\mathbb{R};L^2(\mathbb{R}^d)),
\end{align*}
and for $r,s\leq 0$ we define by duality
\begin{align*}
H^{r,s}(\mathbb{R}\times\mathbb{R}^{d})=(H^{-r,-s}(\mathbb{R}\times \mathbb{R}^d))'.
\end{align*}

By $H^{r,s}(Q)$ we denote the space of restrictions of elements of $H^{r,s}(\mathbb{R}\times\mathbb{R}^d)$ to $Q$ equipped with the quotient norm. The spaces $H^{r,s}(\mathbb{R}\times\partial\Omega)$ and $H^{r,s}(\Sigma)$ are defined analogously. For smooth $\partial\Omega$ they are defined for all $r,s$ whereas for Lipschitz boundaries $\partial\Omega$ they are intrinsically defined for $|r|\leq 1$. This is because the spaces 
are invariant under Lipschitz coordinate transformations only in case $|r|\leq 1$ ($s$ is arbitrary). 

Following Costabel \cite{costabel} we introduce some additional (non-standard) definitions (c.f. also \cites{brown1, brown2}). We denote the subspaces
\begin{align*}
\tilde{H}^{r,s}(Q)=\{ u\in H^{r,s}((-\infty,T)\times\Omega)|\; u(t,x)=0 \; \textrm{for} \; t<0\} \subset H^{r,s}((-\infty,T)\times\Omega)).
\end{align*}
Moreover, we define
\begin{align}
H^{1,\frac{1}{2}}(Q,\partial_t-\Delta)=\{u\in H^{1,\frac{1}{2}}(Q)| \quad (\partial_t-\Delta) u\in L^2(Q)\}
\end{align}
with the norm $$||u||_{H^{1,\frac{1}{2}}(Q)}+||(\partial_t-\Delta)u||_{L^2(Q)}.$$  Let $\gamma^-: H^{\frac{1}{2},\frac{1}{4}}(\mathbb{R}\times\partial\Omega)\rightarrow H^{1,\frac{1}{2}}(\mathbb{R}\times\Omega)$ be a continuous right inverse of the surjective trace map 
$$
 \gamma : H^{1,\frac{1}{2}}(\mathbb{R}\times\Omega) \rightarrow H^{\frac{1}{2},\frac{1}{4}}(\mathbb{R}\times\partial\Omega).
$$ For $u\in H^{1,\frac{1}{2}}(\mathbb{R}\times\Omega,\partial_t-\Delta)$ we denote by $\gamma_1u\in H^{-\frac{1}{2},-\frac{1}{4}}(\Sigma)$ the continuous linear form on $H^{\frac{1}{2},\frac{1}{4}}(\Sigma)$ defined by 
\begin{align}
\gamma_1 u:\phi\rightarrow b(u,\gamma^{-}\phi)
\end{align}
where 
\begin{align}
b(u,v)=\int\limits_Q(\nabla u\cdot \nabla v-((\partial_t-\Delta)u)v)\,dx\,dt+ \langle \partial_t u,v \rangle.
\end{align}
Note that $\partial_t u \in H^{0,-\frac{1}{2}}(\mathbb{R}\times\Omega)$ and therefore we have a well defined dual pairing
$\langle \partial_t u,v \rangle$ between $\partial_t u \in H^{0,-\frac{1}{2}}(\mathbb{R}\times\Omega)$ and $v \in H^{0,\frac{1}{2}}(\mathbb{R}\times\Omega)$ that extends the usual $L^2$-inner product.
The bilinear form $b$ is then continuous on the space $H^{1,\frac{1}{2}}(\mathbb{R}\times\Omega,\partial_t-\Delta)\times H^{1,\frac{1}{2}}(\mathbb{R}\times\Omega)$. In case  $u,v\in C_0^2(\mathbb{R}\times\overline{\Omega})$ this simplifies to
\begin{align}
b(u,v)=\int\limits_{\Sigma}\partial_nu(t,x)v(t,x)\,dt \,d\sigma.
\end{align}
Furthermore, the map $\gamma_1:H^{1,\frac{1}{2}}(Q,\partial_t-\Delta)\rightarrow H^{-\frac{1}{2},-\frac{1}{4}}(\Sigma)$ is continous and for $u\in C^2(\overline{Q})$ we have $\gamma_1u=\partial_nu|_{\Sigma}$. 
Let $G(t,x)$ be defined as 
\begin{align}\label{G}
 G(t,x)=
 \left\{ 
 \begin{array}{llc}
  (4\pi t)^{-d/2}\exp(-\frac{|x|^2}{4t})      & \text{ if } \quad t>0 & \\
  0 & \text{ if }\quad  t\leq 0.\\
 \end{array}  
 \right.\\ \notag
\end{align}
For sufficiently regular $h$ the single layer potential for the heat equation is defined as follows:
\begin{align}
S(h)(t,x)=\int\limits_0^t\int\limits_{\partial\Omega}G(t-s,x-y) h(s,y)\,dy \,ds
\end{align}
for $(t,x)\in Q$. 
The boundary layer potential operator is defined as 
\begin{align}
V(h)(t,x)=\int\limits_0^t\int\limits_{\partial\Omega}G(t-s,x-y) h(s,y)\,dy \,ds
\end{align}
for $(t,x)\in \Sigma$. Finally the double layer potential is defined as
\begin{align}
D(h)(t,x)=\int\limits_0^t\int\limits_{\partial\Omega}\gamma_1 G(t-s,x-y) h(s,y)\,dy \,ds
\end{align}
for $(t,x)\in Q$.

For the following see \cite{costabel}*{Remark 3.2}.  

\begin{prop} \label{onebeforetheorem}
The single layer potential operator $S$ continuously extends to a map $S: H^{-\frac{1}{2},-\frac{1}{4}}(\Sigma) \to H^{1,\frac{1}{2}}(Q)$.
The boundary layer potential operator extends by continuity to an isomorphism
\begin{align}
V:H^{-\frac{1}{2},-\frac{1}{4}}(\Sigma)\rightarrow H^{\frac{1}{2},\frac{1}{4}}(\Sigma).
\end{align}
\end{prop}

\begin{prop} \label{wp}
The trace map $\gamma: u\rightarrow u|_{\Sigma}$
is continuous and surjective from $\tilde{H}^{1,\frac{1}{2}}(Q)$ to $H^{\frac{1}{2},\frac{1}{4}}(\Sigma)$.
For all $f\in \tilde{H}^{-1,-\frac{1}{2}}(Q)$ and $g\in H^{\frac{1}{2},\frac{1}{4}}(\Sigma)$ there exists a unique $u\in \tilde{H}^{1,\frac{1}{2}}(Q)$ with 
\begin{align}\label{dirichletprob}
&(\partial_t-\Delta)u=f \quad \textrm{in} \quad Q \\& \nonumber
 \gamma u=g \quad \textrm{on} \quad \Sigma. 
\end{align}
In case $f=0$ the solution $u$ is given by $u=S(V^{-1}g)=S(\gamma_1 u)-D(\gamma u)$.
\end{prop} 
\begin{proof}
 This summarizes Lemma 2.4 and Theorem 2.9 as well as Theorem 2.20 and Corollary 2.19(c) in \cite{costabel}.
\end{proof}

The well-posedness of the initial value problem \eqref{INg} can be reduced in the usual way to unique solvability of the inhomogenous problem 
of Prop. \ref{wp}. Since the initial value problem is not discussed in \cite{costabel} we will now add some more detailed explanations.

\begin{lemma} We have the following continuous inclusions:
$$H^{1,\frac{1}{2}}(Q,(\partial_t -\Delta)) \subset H^{1}((0,T),H^{-1}(\Omega)) \subset C([0,T],H^{-1}(\Omega)).$$
In particular the restriction map $\gamma_0$ as defined by 
\begin{align*}
&\gamma_0: u \mapsto u|_{t=0} \\&
\gamma_0: H^{1,\frac{1}{2}}(Q,(\partial_t -\Delta)) \to H^{-1}(\Omega)
\end{align*} is well defined and continuous.
\end{lemma}
\begin{proof}
 Let $u \in H^{1,\frac{1}{2}}(Q,(\partial_t -\Delta))$. Then, in particular, $u \in L^2([0,T],H^1(\Omega))$ and $\partial_t u = \Delta u + g$, where $g \in L^2(Q)$.
 Therefore, $\Delta u \in L^2([0,T],H^{-1}(\Omega))$ and $\partial_t u \in L^2((0,T),H^{-1}(\Omega))$. 
 Consequently, $u \in H^1((0,T),H^{-1}(\Omega)) \subset C([0,T],H^{-1}(\Omega))$. Continuity follows from the implied inequalities
 \begin{align*}
  & \| \partial_t u \|^2_{L^2((0,T),H^{-1}(\Omega))}  \leq \| \Delta u \|^2_{L^2((0,T),H^{-1}(\Omega))} + \| (\partial_t -\Delta) u\|_{L^2(Q)}^2 \\& \leq \|u \|^2_{H^{1,0}(Q)} + \| (\partial_t -\Delta) u\|_{L^2(Q)}^2
 \end{align*}
 and the Sobolev embedding theorem. 
\end{proof}

\begin{lemma} \label{hilf1}
 Let $g \in H^1_0(\Omega)$ and let $w$ be solution of the heat
equation on $(0,T) \times \R^d$ defined by
$$
 w(t,x) = \int_{\Omega} G(t,x-y) g(y) dy.
$$
Then $w \in H^{1,\frac{1}{2}}(Q)$.
\end{lemma}
\begin{proof}
 First we remark that if $h \in C^\infty_0(\R^d \setminus \overline \Omega)$ then 
$$
 h_t(x) = \int_{\R^d} G(t,x-y) h(y) dy
$$ 
is a smooth solution of the heat equation on $\R \times \Omega$ that vanishes for $t<0$. In particular this gives an element in
$H^{1,\frac{1}{2}}(\Omega)$. We think of $g$ as an element in $H^1_c(\R^d)$ using extension by zero. By the above remark we can assume without loss of generality that $\hat g$ vanishes at zero of order two. This can be achieved by subtracting off finitely many functions $h$ as above in such a way that moments of $g$ up to order two vanish.
Now consider the function
$$
 u(t,\cdot) = \int_{\Omega} G(|t| ,x-y) g(y) dy,
$$ 
whose restriction to $Q$ equals $w$.
It is easy to compute the Fourier transform of $u$, and the result is
$$
 \hat u (\tau,\xi) =  \frac{2 \xi ^2}{\xi ^4+\tau ^2} \hat g(\xi).
$$
Since $\hat g$ vanishes of order two at zero the right hand side is bounded for small $\tau$. One sees directly that $(1+|\xi|) \hat u \in L^2$ and 
$(1+|\tau|)^{\frac{1}{2}}\hat u \in L^2$ and therefore
$u \in H^{1,\frac{1}{2}}(\R \times \R^d)$.
\end{proof}

We then obtain
\begin{prop} \label{wpwp}
For any $u_0 \in H^1_0(\Omega)$ and $g\in H^{\frac{1}{2},\frac{1}{4}}(\Sigma)$ there exists a unique $u\in H^{1,\frac{1}{2}}(Q)$ with 
\begin{align}\label{dirichletprob}
&(\partial_t-\Delta)u=0 \quad \textrm{in} \quad Q \nonumber \\& 
u|_{t=0} = u_0 \\& \nonumber
 \gamma u=g \quad \textrm{on} \quad \Sigma.  
 \end{align}
\end{prop}
\begin{proof}
Uniqueness is an immediate consequence of the uniqueness statement \cite{costabel}*{Lemma 2.3}.
To show existence extend $u_0$ by zero to an element in $H^1_c(\R^d)$ and construct a solution $w \in H^{1,\frac{1}{2}}(Q)$ as in Lemma \ref{hilf1}.
Then $\gamma w \in H^{\frac{1}{2},\frac{1}{4}}(\Sigma)$.
Now apply  Prop. \ref{wp} with $f=0$ to construct a solution of the homogeneous heat equation $v \in \tilde{H}^{1,\frac{1}{2}}(Q) $ 
with boundary data $g - \gamma w$.
Then $u=v+w$ solves the above problem.
\end{proof}

\section{Proof of Proposition \ref{main1} and Theorem \ref{main2}} \label{s23}

\subsection{Proof of Proposition \ref{main1}}

\begin{proof}
 For $z \in \mathbb{C}^d$ we write $|z|^2 = \overline{z}^\mathtt{T} z$ for the square of its length and $z^2 = z^\mathtt{T} z$ for the analytic extension of the square of the absolute value on $\mathbb{R}^d$.  \\
First note that the heat kernel admits an analytic extension $\tilde G$ as follows
\begin{align}
 \tilde G(t,z-w)=
 \left\{ 
 \begin{array}{llc}
  (4\pi t)^{-d/2}\exp(-\frac{(z-w)^2}{4 t})      & \text{ if } \quad t>0 & \\
  0 & \text{ if }\quad  t\leq 0\\
 \end{array}  
 \right\}.
\end{align}
Note that $\Re{(z-w)^2} = |x-w|^2-|y|^2$ when $z=x+iy$. Therefore in the open set defined by 
$|y|<|x-w|$ the function $\tilde G(t,z-w)$ is smooth in $t$, and complex analytic in $z$.

For an integrable function $f$ on $\Sigma$ we define
 \begin{align}
\tilde S(f)(t,z)=\int\limits_0^t \int\limits_{\partial\Omega} \tilde G(t-s,z-y) f(s,y)\,dy \,ds.
\end{align}
Since the kernel $k(t,z,(s,y))=G(t-s,z-y)$ is smooth on $\R \times \mathcal{E}(\Omega) \times \Sigma$ 
the operator $f \mapsto \tilde S(f)$ extends by continuity to the space of  distributions supported in $\overline \Sigma$,
 $\mathcal{D}'(\overline \Sigma)$. It maps continuously from $\mathcal{D}'(\overline \Sigma)$ to $C^\infty([0,T] \times \mathcal{E}(\Omega))$ and its range consists of functions that vanish of infinite order at $t=0$ in $C^\infty(\mathcal{E}(\Omega))$. To see this simply note that a distribution $f$ with support in $\overline{\Sigma}$ is compactly supported and thus for any open neighborhood $U$ of $\overline{\Sigma}$ we have the estimate $$| f(h) | \leq C_U \| h |_U \|_{C^k(U)}$$ for all $h \in C^\infty(\R \times \R^d)$. One now simply applies this estimate to the heat kernel.\\
 Applying the $\nabla_{\overline z}$ operator shows that
 $$
 \nabla_{\overline z} \tilde S(f)(t,z) =0
 $$
 since the kernel is holomorphic in $z$. We conclude that the mapping $f \mapsto (\tilde Sf)(T)$ is continuous from $\mathcal{D}'(\overline \Sigma)$ to $\calO(\calEO)$.
If $u \in \mathcal{R}_\Omega$ then, by Prop. \ref{wp} and Prop. \ref{onebeforetheorem}, we have the representation
$$
 u = S(r)|_{t=T}
$$
for some $r \in H^{-\frac{1}{2},-\frac{1}{4}}(\Sigma)$. 
If $\gamma$ denotes the trace map to $\Sigma$ the dual of the restriction operator defines a distribution $v=\gamma^* r$ in $\R \times \R^{d-1}$. Here the canonical extension is from from $H^{-\frac{1}{4}}((0,T),H^{-\frac{1}{2}}(\partial\Omega))$
to $H^{-\frac{1}{4}}(\R,H^{-\frac{1}{2}}(\mathbb{R}^{d-1}))$. Hence, $v$ is a compactly supported distribution with support on $\overline \Sigma$. Thus,
$\tilde S v$ defines a function in $\calO(\calEO)$ that restricts to $S(r)$.
Hence, the function $\tilde S(v)|_{t=T}$ defines an analytic extension of $u$ as required, and we have shown $\mathcal{R}_\Omega \subset \calO(\calEO)$.

It remains to prove optimality of the domain $\calEO$. 
Assume without loss of generality that $T=1$. Let $p\in \mathbb{C}^d \setminus \overline{\calEO}$. Then necessarily we have
$\mathrm{dist}(\Re(p),\overline{\Omega}^c)< |\Im(p)|$. This means there exists a point $x_0 \in \overline{\Omega}^c$ with
$|\Re(p) - x_0 | < |\Im (p)|$.
Now suppose $a \in \C$ is such that $\Re(a)>0$. We are going to use the following distributional source 
$$
h_a(x,s) =\chi_{(0,1)}(s)\delta_{x=x_0}e^{-\frac{a}{4-4s}} (1-s)^{\frac{d}{2}-1}.
$$
This gives rise via $S(h_a)$ to a function $u_t(x)$ that is a solution of the heat equation in $\R \times \Omega$ and that extends smoothly across the boundary of $\partial \Omega$. It is given explicitly by
$$
 u(t,z) = \frac{1}{(4 \pi)^\frac{d}{2}} \int_{0}^t (t-s)^{-\frac{d}{2}} e^{-\frac{(z-x_0)^2}{4(t-s)}} e^{-\frac{a}{4(1-s)}}  (1-s)^{\frac{d}{2}-1}ds.
 $$
 Thus, the function $g(z)=u(1,x)$ is reachable. We obtain rather explicitly
 $$
 g(z) = \frac{1}{(4 \pi)^\frac{d}{2}}  \int_{0}^1 s^{-1} e^{-\frac{(z-x_0)^2}{4s}}  e^{-\frac{a}{4s}} ds = \frac{1}{(4 \pi)^\frac{d}{2}} E_1\left(\frac{1}{4} \left((z-x_0)^2+a\right)\right),$$
 where $E_1(z)=\Gamma(0,z)$ denotes the generalized exponential integral \cite{olver2010nist}*{\mathsection 8.19} that can be expressed in terms of the incomplete Gamma function $\Gamma(b,z)$. We have by \cite{olver2010nist}*{\mathsection 8.19(iv)} the following expansion
 $$
  E_1(z) = - \gamma - \log(z) - \sum_{k=1}^\infty \frac{ (-z)^k}{k (k!)},
 $$
 where $\gamma$ is the Euler-Mascheroni constant.
 This function is not analytic at $z=0$, hence, $g(z)$ is not holomorphic at $p$ when $(p-x_0)^2 + a =0$, i.e.
 when $$ \Re(a) = (\Im(z))^2 -(\Re(p) -x_0)^2$$ and $$2 \Im(p) \cdot x_0 = -\Im(a).$$ The condition $(\Im(z))^2 -(\Re(p) -x_0)^2>0$ allows us to find such an $a$ with $\Re(a)>0$.
\end{proof}

\subsection{The one dimensional case as a special example}\label{1dsection}
We consider the problem for the interval $[-L,L]$, which is given by 
\begin{align} \label{intC} 
& \partial_t u=\Delta u \quad \mathrm{in} \quad [0,T] \times (-L,L) \\& \nonumber
u(t,-L)=h_1(t) \quad \quad u(t,L)=h_2(t) \quad \mathrm{in}\quad [0,T]\\& \nonumber
u(0,x)=0 \quad \mathrm{in} \quad (-L,L)
\end{align}
Let $h_1, h_2\in H^{1/4}(0,T)$. 
Because the boundary is a collection of $2$ points we expect layer potential theory tells us that the boundary integral is just evaluation along these two points. In this case the solution is: 
\begin{align}\label{rep1}
u(t,x)= \frac{1}{\sqrt{4 \pi}}\int\limits_0^t\left(q_1(s)\frac{e^{-\frac{|x-L|^2}{4(t-s)}}}{\sqrt{t-s}}+q_2(s)\frac{e^{-\frac{|x+L|^2}{4(t-s)}}}{\sqrt{t-s}}\right)\,ds.
\end{align}
The Fourier transform  of $\chi_{(0,\infty)}(t)\frac{e^{-\frac{L^2}{ 4t}}}{\sqrt{t}}$
is $\frac{e^{-2L\sqrt{\rmi\tau}}}{\sqrt{2\rmi \tau}}$.
Then we solve for the Fourier transform of $q_1$ and $q_2$ using the system
\begin{align}
\begin{pmatrix}
(2\rmi\tau)^{-\frac{1}{2}} & (2\rmi\tau)^{-\frac{1}{2}}e^{-2L\sqrt{\rmi\tau}}\\
(2\rmi\tau)^{-\frac{1}{2}}e^{-2L\sqrt{\rmi\tau}}& (2\rmi\tau)^{-\frac{1}{2}}
\end{pmatrix}
\begin{pmatrix}
\mathcal{F}_tq_1(\tau)\\
\mathcal{F}_tq_2(\tau)
\end{pmatrix}
=\begin{pmatrix}
\mathcal{F}_th_1(\tau)\\
\mathcal{F}_th_2(\tau)
\end{pmatrix}
\end{align}
This system is invertible for all $\tau\geq 0$ as the determinant of the coefficient matrix is 
\begin{align}
(2\rmi\tau)^{-1}(1-e^{-2L\sqrt{\rmi\tau}})\neq 0.
\end{align}
We define the analytic extension of $u(t,z)$ as before. Let
 \begin{align}
\tilde S(h)(t,z)=\nonumber \frac{1}{\sqrt{4\pi}}\int\limits_0^t\left(q_1(s)\frac{e^{-\frac{(z-L)^2}{4(t-s)}}}{\sqrt{t-s}}+q_2(s)\frac{e^{-\frac{(z+L)^2}{4(t-s)}}}{\sqrt{t-s}}\right)\,ds
\end{align}
 and the same proof follows with $\Sigma=(0,T)\times \{-L,L\}$.  

\subsection{Proof of Theorem \ref{main2}}

\begin{proof} 
Assume that $\Omega$ is a ball $B_R(x_0)$. We can assume without loss of generality that $x_0=0$. 
Then $\calEO$ is simply $$\calEO = \{ z = x + \rmi y \in \C \mid  |x| + |y| < R \}.$$ This domain is therefore invariant under Wick rotation, i.e.
$\rmi \calEO = \calEO$ and this is the property that we are going to use.
In particular, the fibre of $0 \in \Omega$
in the ball bundle $\calEO$ is exactly $\rmi \Omega$. 
Now assume that $u \in \calO(U)$ for some bounded open set  $U \subset \C^d$ with $\overline{\calEO} \subset U$. Fix a subset $U_1$
with $\overline{\calEO} \subset U_1$ and $\overline{U_1} \subset U$ and pick a cutoff function $\chi \in C^\infty_0(U)$ with $\chi(x) =1$ for all
$x \in U_1$.  In the following we identify the complex plane notationally with $\R^2$ with $\C$ and therefore use the notations $u(x,y)$ and $u(x+ \rmi y)$ interchangeably.
Since $u(x,y)$ is holomorphic it satisfies the Cauchy Riemann equations $(\nabla_x + \rmi \nabla_y) u(x,y) =0$ on $U$.
Let $K_{t,y}$ be the solution operator for the heat equation on $\R^d_y$, which is a standard convolution with $G(t,y)$ defined by \eqref{G} as its kernel. 
Define $\phi_t(y) = K_{t,y}(u \chi)(0,y)$.
The function $\phi_t(y)$ is in $C\left([0,\infty)_t,C^\infty(\R^d_y ) \right) \cap C^\infty([0,\infty)_t \times \R^d_y )$ and is a solution of
\begin{align*}
&\left(\partial_t -\Delta_{y}\right) \phi_t(y)=0 \\&
\phi_0(y)= \phi(y)=(u\chi)(iy).
\end{align*}
For $t>0$ the function $\phi_t(y)$ is obtained by applying an integral operator with entire kernel in $y$ to a compactly supported function. Hence, $\phi_t(y)$ is for any $t>0$ an entire function in the $y$-variable. Let us denote the analytic extension in the $y$-variable by $u_t(x, y)$. Here the roles of $x$ and $y$ are interchanged, so let us explain in more detail what this means. The function $u_t(x, y)$ satisfies the Cauchy Riemann equations
$(\nabla_{x} + \rmi \nabla_y) u_t(x, y) = 0$ and is completely determined for $t>0$ by $u_t(0,y) = \phi_t(y)$, e.g. $u_t(z)=\phi_t(-iz)$. 
Lemma \ref{LemmaLemma} now implies that $u_t$ converges to $u(x,y)$ uniformly on $\calEO=\rmi \calEO$ as $t \to 0_+$. Indeed,  choose $R'>R$ so that $\rmi B_{R'}(0) \subset U_1$ and Theorem \ref{main3}  gives us uniform convergence on the compact set $\overline{\calEO}$.
By construction $(\Delta_x + \Delta_y)u_t(x,y) =0$ and therefore  $u_t(x,y)$ solves the inverse heat equation
\begin{align*}
&(\partial_t+\Delta_x)u_t(x,y)=0\\&
u_0(x,y)=u(x,y)
\end{align*}   
on $\rmi \calEO$, the ``Wick rotated'' $\calEO$. Here we have used uniqueness of the analytic continuation to conclude that also $u_t(x,y)$ solves the heat equation.
Since $\rmi \calEO$ contains $\Omega$ the function $u_t(x):= u_t(x,0)$ solves
\begin{align*}
&(\partial_t+\Delta_x)u_t(x)=0 \,\, \textrm{in } Q\\&
u_0(x)=u(x) \\&
u_t(x)|_{\Sigma}=h(t,x) 
\end{align*} 
where $h$ extends smoothly across $\partial \Omega$.
Change of variables $T-t\rightarrow t$ shows that $u \in \mathcal{R}(T,g)$ for some $T>0$ and some function $g(x) = u_T(x)$.
\end{proof}

\section{Analytic properties of solutions of the heat equation} \label{s5}

In this section we will prove Theorem \ref{main3} which is a major ingredient in the proof of Theorem \ref{main2}. By Proposition \ref{main1} any positive time solution of the heat equation is analytic in $\mathcal{E}(\Omega)$. We investigate in this section what happens if the function $u$ was already analytic in $\mathcal{E}(\Omega)$ at time $t=0$.
Assume that $\Omega$ is a bounded open subset in $\R^d$ and $\Omega_1$ is a bounded open subset such that $\overline{\Omega} \subset \Omega_1$. It is easy to see that $\overline{\mathcal{E}(\Omega)} \subset \mathcal{E}(\Omega_1)$. We can therefore find a smooth 
cutoff function $\chi \in C^\infty_0(\C^d)$ with support in $\mathcal{E}(\Omega_1)$ and which is equal to one on $\overline{\mathcal{E}(\Omega)}$.
We have the following Lemma for the free heat operator $K_t= e^{t \Delta_0}$ on $\R^d$.

\begin{lemma} \label{LemmaLemma}
 Assume that $u \in C^\infty(\Omega_1)$ has an analytic extension to $\mathcal{E}(\Omega_1)$. Then the analytic continuation of
 $K_t(u \chi)$ converges to the analytic continuation of $u$ uniformly on compact subsets of  $\mathcal{E}(\Omega)$ as $t \to 0_+$.
\end{lemma}
\begin{proof}
 Let us denote the analytic extension of $u$ by the same letter, i.e. $u(w)$ makes sense for $w \in \mathcal{E}(\Omega_1)$.
 Fix a point $z \in \mathcal{E}(\Omega)$, i.e. $| \Im(z) | < \mathrm{dist}(\Re(z),\Omega^c)$. The explicit formula for $K_t(u \chi)$ is
 $$
  K_t(u \chi)(z) = \frac{1}{(4 \pi t)^{\frac{d}{2}}} \int_{\R^d} e^{-\frac{(z-w)^2}{4t}}g(w) dw,
 $$
 where $g(w)= u(w) \chi(w)$. This integral is thought of as an integral over the real submanifold in the complex domain and we will now shift the contour in the $\Im(z)$-direction.  Let $\Gamma$ be the set $\R^d + \rmi \Im(z)$ and let $\mathcal{T}$ be the region
 $\R^d + \rmi [0,1] \Im(z) = \left \{ w + \rmi t \Im(z) \mid w \in \R^d, t \in [0,1] \right \}$. Recall that for a smooth function $\varphi$ in the complex plane we have the formula
 $$
  \int_{\partial Y} \varphi(x + \rmi y) d(x + \rmi y) = 2 \rmi \int_{Y} \overline{\partial}_z \varphi(x+\rmi y) dx dy,
 $$
 if $Y$ is a region with $C^1$-boundary (c.f equation 3.1.9 in \cite{hormander}). Thus, shifting the contour in the direction of $\Im(z)$, we obtain
 \begin{gather*}
  K_t(u \chi)(z) = \frac{1}{(4 \pi t)^{\frac{d}{2}}} \int_{\R^d} e^{-\frac{(z-w)^2}{4t}}g(w) dw \\
  = \frac{1}{(4 \pi t)^{\frac{d}{2}}} \int_{\Gamma} e^{-\frac{(z-w)^2}{4t}} g(w) dw  + 2 \rmi 
   \frac{1}{(4 \pi t)^{\frac{d}{2}}} \int_{\mathcal{T}} e^{-\frac{(z-w)^2}{4t}} u(w) \overline{\partial}_w \chi(w) d \sigma(w) \\= I_1(z) + I_2(z)
 \end{gather*}
 where $d  \sigma$ is the Lebesgue measure on  $\mathcal{T}$.
 The first integral $I_1(z)$ equals
 $$
  I_1(z) = \frac{1}{(4 \pi t)^{\frac{d}{2}}} \int_{\R^d} e^{-\frac{(\Re(z)-w)^2}{4t}} g(w + \rmi \Im(z)) dw. 
 $$
 Since the heat kernel is a $\delta$-family (this is a standard good kernel argument) this integral converges to 
 $g(\Re(z) + \rmi \Im(z))=g(z)$ at $t \searrow 0$. Since $g$ is smooth and compactly supported the convergence is uniform on compact subsets. It remains to show that the second integral
 $I_2(z)$ converges to zero uniformly on compact subsets.  Let $Z_z$ be the compact set $(\supp \overline{\partial}_w \chi) \cap \mathcal{T}$. Then the integrand in $I_2$ has support in $Z_z$, i.e. we can restrict the integration over $Z_z$ by the support properties of $\overline{\partial}_w\chi$.  By compactness of $\supp \overline{\partial}_w \chi$ there exists an $\epsilon_1>0$ such that for all 
 $z \in \mathcal{E}(\Omega)$ and all $w \in Z_z$ we have
\begin{gather*}
 \Im w = t \Im z \textrm{ for some } t \in [0,1],\\
 \mathrm{dist}(\Re(w),\Omega^c) + \epsilon_1  < | \Im(w) |  \leq | \Im(z) | < \mathrm{dist}(\Re(z),\Omega^c).
\end{gather*}
Note here $\epsilon_1$ is independent of $z$.

 For all elements $w \in Z_z$ we therefore have the inequality
 \begin{gather*}
  | \Im(z) -\Im(w) | =| \Im(z)| - |\Im(w)| \leq  \mathrm{dist}(\Re(z),\Omega^c) - \mathrm{dist}(\Re(w),\Omega^c) - \epsilon_1\\
  \leq \mathrm{dist}(\Re(z),b) - \mathrm{dist}(\Re(w),b) - \epsilon_1 \leq |\Re(z) - \Re(w)| - \epsilon_1,
 \end{gather*}
 where $b$ is a point on the boundary of $\Omega$ with $\mathrm{dist}(\Re(w),\Omega^c)=\mathrm{dist}(\Re(w),b)$.
 We have used in the first step that $\Im w = t \Im z$ for some $0 \leq t \leq 1$, and in the last step the reverse triangle inequality.
 The statement now follows since in the region 
 $|\Im(z-w)| + \epsilon_1 < |\Re(z-w)|$ the function $e^{-\frac{(z-w)^2}{4t}}$ vanishes of infinite order at $t=0$ uniformly.
 \end{proof}

\begin{proof}[Proof of Theorem \ref{main3}] 
Assume $u \in C^\infty(\Omega)$ is a smooth function that extends to a holomorphic function on $\mathcal{E}(\Omega)$.  Assume that $g$ is any solution of the heat equation on $Q = [0,T] \times \Omega$
 with $g(0)=u |_{\Omega}$. We need to show that the analytic continuation of $g(t)$ for $t>0$ converges uniformly on compact subsets of $\mathcal{E}(\Omega)$ to $u$ as $t \to 0_+$.
Fix a compact subset $K \subset \mathcal{E}(\Omega)$ and
 choose an open subset with smooth boundary $\Omega' \subset \Omega$ with $\overline{\Omega'} \subset \Omega$ so that
 $K \subset \mathcal{E}(\Omega')$. Such a subset can be constructed as follows. First note it follows from compactness of $K$ that we can find a constant $\epsilon_2>0$ such  that for all $z \in K$ we have $\Im(z) < \mathrm{dist}(\Re(z),\partial \Omega) -\epsilon_2$.   
 Now simply find an $\epsilon_2/2$-approximation of the Lipschitz domain by a smooth domain from inside. Such approximations are well known to exist (see for example, \cite{Verchota}). 

Let $\chi$ be a cutoff function as in the above Lemma \ref{LemmaLemma} which is compactly supported in $\calEO$ and equals to one on $\mathcal{E}(\Omega')$. Now define $f(t) := K_t(u \chi)(z)$. Then, by Lemma \ref{LemmaLemma}
 the function $f(t)$ has an analytic extension to $\mathcal{E}(\Omega')$ that converges uniformly on $K$ to $u$.
 Now consider the function 
 $$
 h(t)= \left \{ \begin{matrix} g(t) - f(t) & t \geq 0,\\ 0 & t<0.\end{matrix} \right.
 $$

Thus, we have a smooth solution $h$ of the heat equation on $\Omega'$ that vanishes at zero. Therefore, since for smooth boundaries $SV^{-1}$ maps smooth functions to smooth functions (\cite{costabel}*{Section 4}), there exists smooth data $r \in C^\infty((0,T)\times\partial\Omega')$ such that
 $$
  h(t,z) = \int_0^t \int\limits_{\Gamma} \tilde G(t-s,z-y) r(s,y)\,dy \,ds.
 $$
 Since for all $z \in K$ and $y \in \partial \Omega'$ we have $|\Im(z-y)| < |\Re(z-y)| - \epsilon_3$ for some $\epsilon_3>0$. Thus, $h$ converges uniformly to zero on $K$ as $t \to 0_+$.
 \end{proof}

\emph{Acknowledgments.}
Part of the work for this paper was carried out during the program
``Randomness, PDEs and Nonlinear Fluctuations'', funded by the \emph{Hausdorff Center of Mathematics} in Bonn, and A.W. is grateful for the support and hospitality during this time.

\end{document}